\documentclass[12pt]{amsart}
 \usepackage{graphicx,mdframed}
\usepackage{amsmath,amsfonts,amscd,amssymb,mathrsfs,hyperref,cite}
\usepackage[all]{xypic}
\usepackage[enableskew]{youngtab} % For young tableaux
\usepackage{ytableau}

\usepackage{lscape,scalefnt}
\usepackage[margin=1in]{geometry} 
\usepackage{color}

\newcommand\scalemath[2]{\scalebox{#1}{\mbox{\ensuremath{\displaystyle #2}}}}

\numberwithin{equation}{section}

\newtheorem{theorem}{Theorem}[section]
\newtheorem{computation}[theorem]{Computation}

\newtheorem{lemma}[theorem]{Lemma}

\newtheorem{prop}[theorem]{Proposition}

\theoremstyle{definition}

\newtheorem{example}[theorem]{Example}
\newtheorem{remark}[theorem]{Remark}

\def\E{{\mathcal E}}

\def\H{{\mathcal H}}
\def\L{{\mathcal L}}
\def\M{{\mathcal M}}

\def\P{{\mathcal P}}
\def\W{{\mathcal W}}

\def\G{{\mathbb G}}

\newcommand{\defi}[1]{\textsf{#1}} % for defined terms

\newcommand{\SL}{\operatorname{SL}}

\newcommand{\GL}{\operatorname{GL}}

\newcommand{\PP}{\mathbb{P}}
\newcommand{\NN}{\mathbb{N}}
\newcommand{\RR}{\mathbb{R}}
\newcommand{\CC}{\mathbb{C}}

\newcommand{\QQ}{\mathbb{Q}}

\newcommand{\Seg}{\operatorname{Seg}}

\def\bw#1{{\textstyle \bigwedge^{\hspace{-.2em}#1}}}
\def\o{{\otimes}}

\title[Computations and Equations for Segre-Grassmann Hypersurfaces]{Computations and Equations for \\ Segre-Grassmann Hypersurfaces}

\author[Daleo, Hauenstein, Oeding]{Noah S. Daleo, Jonathan D. Hauenstein, Luke Oeding}
\address[ndaleo@worcester.edu]{Noah S. Daleo, Department of Mathematics, Worcester State University, Worcester, MA.}
\urladdr{\url{http://www.worcester.edu/noah-daleo/}}
\address[hauenstein@nd.edu]{Jonathan D. Hauenstein, Department of Applied and Computational Mathematics and Statistics, University of Notre Dame, Notre Dame, IN. \\
\thanks{Research of Daleo and Hauenstein supported in part by NSF grant DMS-1262428 and DARPA YFA. }}
\urladdr{\url{www.nd.edu/~jhauenst}}
\address[oeding@auburn.edu]{Luke Oeding, Department of Mathematics and Statistics, Auburn University, Auburn, AL. \\
\thanks{Oeding thanks the organizers of NIMS (National Institute for Mathematical Science), Daejeon, S. Korea, for their hospitality and support during the preparation of this manuscript.}}
\urladdr{\url{www.auburn.edu/~oeding}}
\date{\today}

\begin{document}

\begin{abstract}
In 2013, Abo and Wan studied the analogue of Waring's problem for systems of skew-symmetric forms and identified several defective systems.  
Of particular interest is when a certain secant variety of a Segre-Grassmann variety is expected to fill the natural ambient space, but is actually a hypersurface.  
Algorithms implemented in {\tt Bertini} \cite{Bertini} are used to determine the degrees of several of these hypersurfaces, and representation-theoretic descriptions of their equations are given.
We answer \cite[Problem 6.5]{AboWan},  and confirm their speculation that each member of an infinite family of hypersurfaces is minimally defined by a (known) determinantal equation. 
While led by numerical evidence, we provide non-numerical proofs for all of our results.
\end{abstract}



\maketitle
\section{Introduction}
Secant varieties, while a classical topic in algebraic geometry, have received much attention over the past several years largely due to the vast number of applications to many fields such as Geometric Complexity Theory and Signal Processing (e.g., see \cite{LandsbergGCT} and \cite{SahnounComon}).

Suppose $X$ is an algebraic variety in $\PP^{N}$, and for simplicity, assume
that $X$ is not contained in any linear subspace. 
The \emph{$X$-rank} of a point $[p]\in \PP^{N}$ is the minimum number $r$ such that $p = \sum_{i=1}^{r}x_{i}$ with $[x_{i}]\in X$.  The Zariski closure of the points of $X$-rank $r$ is the \emph{$r$-secant variety to $X$}, denoted $\sigma_{r}(X)$.\footnote{Note that taking the Zariski closure often causes a failure of upper semi-continuity of $X$-rank, for instance in the case of tensors of order 3 or more.} We say that the points of $\sigma_{r}(X)$ have \emph{$X$-border rank $r$}.
For tensors and related algebraic varieties, $X$-rank and $X$-border rank provide a useful perspective; see \cite{BucLan_ranks}. The reader may find the recent lecture notes \cite{CarliniGrieveOeding} to be useful for general background on secant varieties, as well as an extensive list of references contained~therein.

The first question one asks about $X$-rank for $X \subset \PP^{N}$ is which $X$-border rank fills the ambient space $\PP^{N}$. Indeed, the famous Alexander-Hirschowitz Theorem \cite{AH92,alexander1995polynomial} answered this question when $X$ is the Veronese embedding of projective space (see also \cite{OttBra08_AH, Postinghel} for modern accounts).  The analogous question for the Segre embedding of the Cartesian product of projective spaces into the projectivization of a tensor product of vector spaces has been studied, for example in \cite{CGG2_Segre, CGG5_rational, AOP_Segre}.
Many cases were settled, for example in the case of $\PP^{1}$'s in \cite{CGG_P1s}, but this problem is not yet completely solved (see \cite{ChiOttVan} for recent progress). The skew-symmetric version of this question was addressed in \cite{CGG6_Grassman, KarinDraismadeGraaf, AOP_Grassmann}, again with some cases solved and some cases remaining.  

Another question one may ask regarding $X$-border rank is to describe the defining equations of $\sigma_{r}(X)$. From such equations, one can easily decide the $X$-border rank of any given point in $\PP^{N}$.
Versions of this test are extremely important, for instance, in algebraic complexity theory \cite{Lan_matrix_mult, HauIkeLan}. 

The purpose of this paper is twofold.  The first objective is to find equations for secant varieties of certain Segre-Grassmann varieties. 
We focus on two cases where the secant variety in question is a hypersurface. One of these cases solves a problem left open in \cite{AboWan}, while the other case, which is actually an entire family of hypersurfaces, confirms a guess in Abo and Wan's work that an Ottaviani-type construction gives the requisite equations.   The second objective is to demonstrate the power and use of combining tools from Numerical Algebraic Geometry and Representation Theory, which we hope will be used to address many other problems in the future.  While partially skew-symmetric tensors are certainly less studied than the fully symmetric and non-symmetric cases, it is often the case that methods for finding equations for border rank in one symmetry class inform techniques for another.  For instance, Ottaviani's approach to Aronhold's invariant for symmetric tensors as a Pfaffian led to a new construction of Strassen's invariant for non-symmetric tensors \cite{Ottaviani09_Waring, LanOtt11_Equations}. 

Here is an outline of the rest of this paper. 
Section~\ref{Sec:Notation} contains notation and background information. 
Sections~\ref{Sec:Degree} and~\ref{sec:rep} describe the algorithms used from Numerical Algebraic Geometry and  Representation Theory, respectively, with Theorem~\ref{thm:sigma5} answering \cite[Problem~6.5]{AboWan}. 
In Section~\ref{Sec:sigma3l+2}  we consider an infinite family of hypersurfaces and show that known determinantal equations define them (Theorem~\ref{thm:family}). 
In Section~\ref{sec:K} we study the irreducibility of a determinant of the tensor product of two skew-symmetric matrices, which we use in the proof of Theorem~\ref{thm:family}.

\section{Notation and preliminaries}\label{Sec:Notation} Let $\bw{k+1} \CC^{n+1}$ denote the vector space of alternating $k+1$ forms on an $n+1$ dimensional (complex) vector space, whose natural basis is given by the pure wedge products  \mbox{$e_{j_{1}}\wedge\dots\wedge e_{j_{k+1}}$}, with $1\leq j_{1}<\dots < j_{k+1} \leq n+1$ and $\{e_{j} \}$  a basis of~$\CC^{n+1}$.
We now consider $\CC^{m+1}\otimes \bw{k+1} \CC^{n+1}$ consisting of partially skew-symmetric tensors.
We will write $\{x_{i,j_{1},\dots,j_{k+1} } \}$ for coordinates on  $\CC^{m+1} \otimes \bw{k+1}\CC^{n+1}$, where $1\leq i\leq m+1$ and $1\leq j_{1}< \dots < j_{k+1}\leq n+1$. 
By slicing in the first tensor mode, a point in this space may be thought of as an $m+1$-dimensional system of alternating $k+1$ forms on $n+1$ variables.  It is natural to consider the points of rank $1$ to be those points which are ``pure tensors'' or ``indecomposable tensors'' with the required symmetry.

Let $X =\Seg(\PP^{m} \times \mathbb{G}(k,n))$ be a Segre-Grassmann variety, which is the Segre product of a projective $m$-plane and the Grassmann variety of $k$-dimensional projective subspaces of an $n$ dimensional projective space.  
The natural embedding of $X$ is by a Segre-Pl\"ucker embedding into $\PP \left( \CC^{m+1} \otimes \bw{k+1}\CC^{n+1} \right)$.  
A general point on $\Seg(\PP^{m} \times \mathbb{G}(k,n))$ is~(a~pure~tensor)~of~the~form
\[
[v \otimes (w_{0}\wedge \dots \wedge w_{k})]
,\]
where $[v] \in \PP^{m}$, and $w_{0},\dots, w_{k}$ form a basis of a $k$-dimensional (projective) linear subspace of $\PP^{n}$.  
Let $\sigma_{s}(\Seg(\PP^{m} \times \mathbb{G}(k,n)))$ denote the $s$-th secant variety of the Segre-Grassmann variety.
A general point on this variety is of the form
\begin{equation}\label{param}
\left[\sum_{i=1}^{s}v^{i} \otimes (w_{0}^{i}\wedge \dots \wedge w_{k}^{i}) \right]
,\end{equation}
where the superscripts are just formal placeholders and the other terms have the same interpretation as before. Thus, the points of $X$-rank $s$ in $\CC^{m+1}\otimes \bw{k+1} \CC^{n+1}$ may be thought of as those points which have the interpretation as a formal linear combination of~$s$ terms, each term being an $(m+1)$-dimensional system of $k$-planes~in~$\PP^{n}$.

Here is a straightforward way to obtain coordinates for the points (and hence a parametrization of the variety).
Let $v = (v_{0},\dots, v_{m})$, and let $E = (e_{i,j})$ be a $(k+1) \times(n+1)$ matrix.
One obtains an $(m+1) \times \binom{n+1}{k+1}$ vector for a point on  $\Seg(\PP^{m} \times \mathbb{G}(k,n))$ as
\[
\left(v_{i}\cdot \Delta_{I}(E)\right)_{i,I} ,
\]
where $\Delta_{I}$ is the maximal minor of $E$ described by the columns of  $I~=~(i_{1},\dots, i_{k+1})$.
Moreover, one may generate random points on $\sigma_{s}(\Seg(\PP^{m} \times \mathbb{G}(k,n)))$ by letting $v$ and $E$ be (respectively) a random vector and a random matrix, and summing $s$ random points of $\Seg(\PP^{m} \times \mathbb{G}(k,n))$.

The main tool for determining the dimension of a secant variety is the well-known Terracini lemma. For an algebraic variety $X\subset \PP^{N}$,
 and if $[x]\in X$ is a smooth point, let $\widehat{T_{x}}X$ denote the cone over the tangent space of $X$ at $[x]$.
\begin{lemma}[Terracini]
Let $X\subset \PP^{N}$ be an algebraic variety, and let $[x_{1}],\dots,[x_{k}]$ be general points of $X$. Set $p = \sum_{i=1}^{k}x_{i}$ and suppose that $[p]$ is a general point of $\sigma_{k}(X)$.  Then the tangent space of the secant variety is the sum of tangent spaces to the original variety:
\[
\widehat{T_{p}}\sigma_{k}(X) = \widehat{T_{x_{1}}}X +\dots + \widehat{T_{x_{k}}}X
.\]
\end{lemma}

If $X$ is an $k$-dimensional algebraic variety in $\PP^{n}$, one expects (by Terracini's lemma) that its $r$-th secant variety $\sigma_{r}(X)$ should have dimension $\min\{r(k+1)-1, n\}$.  Abo and Wan  \cite{AboWan} classified many cases of defective Segre-Grassmann varieties, and here is one of their results, which follows from\cite[Thm.~5.3]{AboWan} and their discussion in \cite[Section~6]{AboWan}.

\begin{prop}[\cite{AboWan}]
$\sigma_{5}(\Seg(\PP^{2} \times \mathbb{G}(2,5)))$ is a hypersurface in~$\PP^{59}$.
\end{prop}

\subsection{Symmetry}
Let $V \cong \CC^{m+1}$ and $W \cong \CC^{n+1}$. 
The Segre-Grassmann variety $\Seg(\PP V \times \mathbb{G}(k,\PP W))$ is left invariant under the action of $\GL(V) \times \GL(W)$.  Its secant variety inherits the same symmetry.  Moreover, the graded coordinate ring 
\[\CC[V\otimes \bw{k+1} W] = \bigoplus_{d\geq 0} S^{d}(V\otimes \bw{k+1} W)^{*} \]  
also inherits this symmetry.  
A consequence of Schur-Weyl duality is that each degree $d$ piece decomposes as
\begin{equation}\label{coordRing}
S^{d}(V\otimes \bw{k+1} W)^{*} = \bigoplus_{\lambda\vdash d,\;\;\pi\vdash(k+1)d} S_{\lambda}V^{*} \o S_{\pi}W^{*} \otimes \CC^{[\lambda,\pi]}
,\end{equation}
where  $S_{\lambda}V^{*}$ and  $S_{\pi}W^{*}$ are Schur modules and $\CC^{[\lambda,\pi]}$ is the multiplicity space associated to the partitions $\lambda,\pi$.

This decomposition may be obtained via a character computation.  This computation is conveniently carried out in the program LiE \cite{LiE} (see Section~\ref{sec:rep} for an example).
An explicit basis of  $\CC^{[\lambda,\pi]}$ may be obtained by a careful application of Young symmetrizers.  We will explain this construction in 
Section~\ref{sec:rep}.
The following section uses numerical algebraic geometric algorithms  to determine the degree of this hypersurface and several other related ones.  These degrees are used as input to determine an equation defining each hypersurface, using Representation Theory in Section~\ref{sec:rep} and careful multi-linear algebra in Sections~\ref{Sec:sigma3l+2},\ref{sec:K}.

\section{Computing the degree of a hypersurface with {\tt Bertini}}\label{Sec:Degree}

Computing the degree and defining equation for a parametrized hypersurface is a classical problem in elimination theory (e.g., see \cite[Chap. 3]{CLO_text}).   
For this we turn to Numerical Algebraic Geometry, namely techniques in {\em numerical elimination theory}  \cite{HauSom_Witness,HauSom_Membership} summarized in \cite[Chap. 16]{BertiniBook}. We use such numerical techniques to  compute the degree of each hypersurface in our study. Once the degree is known, we then use Representation Theory and Linear Algebra, in Sections \ref{sec:rep} and \ref{sec:K}, to compute the defining equation for each hypersurface.    

Before describing in detail the computation involving  $\sigma_5 (\mathrm{Seg}(\PP^{2} \times \mathbb{G}(2,5)))$, we first summarize the procedure from a geometric point of view.   Suppose that $\H\subset\PP^n$ is an irreducible hypersurface. Since $\deg \H = |\H\cap\L|$ for a general line $\L\in\G(1,n)$, one simply needs to compute the finite set of points $\W = \H\cap\L$, called a {\em witness point set} for $\H$ (see \cite[Chap.~13]{SomeseWampler05}).

To compute $\W$, we first generate a point in $\W$.   In our case, we have a parametrization of $\H$  so it is trivial to compute a smooth point $x\in\H$.  We then choose $\L$ to be a general line passing through $x$, where $x\in\W = \H\cap\L$.

Starting from one point in $\W$, we then use random {\em monodromy loops} \cite{SVW02} to attempt to generate additional points in $\W$.   We first select a random path $\M:[0,1]\rightarrow\G(1,n)$ with $\M(0) = \M(1) = \L$.  Then, for each $w\in W$, we track the path $p_w(t)\in \H\cap\M(t)$ with $p_w(0) = w$ to compute the point $p_w(1)\in\W$.  

As stated, such random monodromy loops allow one to potentially generate additional points in $\W$ without a definitive criterion for when we have computed all points in $\W$. 
A heuristic criterion is when several of such loops fail to generate new points. 
The definitive criterion we will use is the {\em trace test} \cite{SVW02}, which is performed as follows. 
Let $\P:\RR\rightarrow\G(1,n)$ be a family of lines that are parallel with respect to some affine coordinate chart such that $\P(0) = \L$ and $\W'\subset \W$. 
Then, $\W' = \W$ if and only if 
\begin{equation}\label{eq:tracetest}
  \hbox{every coordinate of~} \sum_{w\in \W'} p_w(t) \hbox{~is linear in $t$},
\end{equation}
where $p_w(t)\in \H\cap\P(t)$ with $p_w(0) = w$.  
Since two distinct points define a unique line, we test this linearity 
condition in practice by testing if three points lie on a line,
namely the three points corresponding to $t=-1,0,1$.
If this linearity test fails, then $\W'\subsetneq\W$ and we must perform more monodromy loops to compute the missing points. 
This procedure is summarized Figure \ref{alg:trace}.

\begin{figure*}[!t]\caption{Summary of procedure for computing $\deg\H$}\label{alg:trace}
\medskip 
\begin{mdframed}
Let $\H$ be an irreducible hypersurface and $\L$ be a line so that $\deg \H = |\H\cap\L|$.  
\begin{enumerate}
\item Generate a point $x\in\H\cap\L$.  Initialize $\W := \{x\}$.
\item Perform a random monodromy loop starting at the points in $\W$:
\begin{enumerate}
\item Pick a random loop $\M(t)$ in the space of lines so that $\M(0) = \M(1) = \L$.  
\item Track the curves $\H\cap\M(t)$ starting at the points in $\W$ at $t = 0$ to compute the endpoints $\E$ at $t = 1$.  (Hence, $\E\subset\H\cap\L$).
\item Update $\W := \W \cup\E$.  
\end{enumerate}
\item Repeat (2) until the trace test performed at $t=-1,0,1$ verifies the linearity condition (\ref{eq:tracetest}) so that $\W = \H\cap \L$.
\end{enumerate}
Upon completion of this algorithm, we have $\deg\H=|\W|$.
\end{mdframed}
\end{figure*}

We need to modify this procedure for parametrized hypersurfaces.   This results in a problem in  numerical elimination theory in which computations are performed on the base of the parametrization and witness sets are simply replaced by {\em pseudowitness sets} \cite{HauSom_Witness}.
This approach facilitated by path tracking using {\tt Bertini} \cite{Bertini} yielded the following.
%
%\pagebreak
\begin{computation}\label{thm:degree}
We applied the numerical procedure in Figure \ref{alg:trace} yielding:
\begin{enumerate}
\item\label{deg6} the hypersurface $\sigma_{5}(\Seg(\PP^{2} \times \mathbb{G}(2,5))) \subset \PP^{59}$ has degree 6;
\item\label{deg21}
the hypersurface $\sigma_{5}(\Seg(\PP^{2} \times \mathbb{G}(1,6))) \subset \PP^{62}$  has degree 21;
\item  \label{deg33}the hypersurface
$\sigma_8 (\Seg(\PP^2 \times \mathbb{G}(1,10))) \subset \PP^{164}$ has degree 33;
\item \label{deg45}the hypersurface
$\sigma_{11} (\Seg(\PP^2 \times \mathbb{G}(1,14))) \subset \PP^{314}$ has degree 45.  
\end{enumerate}
\end{computation}
\begin{proof}[Summary of computation.]
In our execution of the procedure for the hypersurface 
$\H = \sigma_5 (\mathrm{Seg}(\PP^{2} \times \mathbb{G}(2,5)))$, it took~$6$ random monodromy loops to compute the six points in $\H\cap\L$.  The total procedure lasted $50$ seconds using a single $2.3$ GHz core of an AMD Opteron 6376 processor.  
The last 3 hypersurfaces come from \cite{AboWan} and are part of an infinite family that will be considered in Section~\ref{Sec:sigma3l+2}.  
In our execution for these hypersurfaces, it took $13$, $12$, and $13$ random monodromy loops to yield the degree many points for each case, respectively.  Using a total of sixteen $2.3$ GHz cores, the total procedure lasted $2.5$ minutes, $32$ minutes, and $5.5$ hours, respectively.
\end{proof}

Computation \ref{thm:degree} gives very strong evidence that the known determinantal equations for these hypersurfaces are actually irreducible and minimally generate the corresponding prime ideal.
Non-numerical proofs of the results of Computation~\ref{thm:degree} as well as generalizations  are provided in Sections~\ref{sec:rep} and \ref{Sec:sigma3l+2}.

\section{Young symmetrizers and explicit polynomial invariants}\label{sec:rep} 

By Computation~\ref{thm:degree}(\ref{deg6}), we know that we are looking for a degree 6 equation for $\sigma_5 (\mathrm{Seg}(\PP^{2} \times \mathbb{G}(2,5)))$. Moreover, by the symmetry of the variety, we know that we are looking for  a degree 6 polynomial invariant for $\SL(3)\times \SL(6)$ acting on $\CC^{3}\o \bw{3}\CC^{6}$. Using \cite{LiE}, we computed the entire isotypic  decomposition of the degree 6 part of the coordinate ring $\CC[\CC^{3}\o \bw{3}\CC^{6}] $ in \eqref{coordRing} above via the LiE command {\verb sym_tensor(6,[1,0]^[0,0,1,0,0],A2A5) }
(which performs a character computation to determine the dimensions of the~multiplicity~spaces).

The output is a long polynomial, but the occurrence of {\verb 1X[0,0,0,0,0,0,0] } tells us, in particular, that the trivial representation occurs with multiplicity one.
Now that we know that there is only one non-trivial degree 6 invariant (up to trivial rescaling), we can apply a Young symmetrizer construction to produce the invariant as follows. We will describe the entire process with the degree 6 Abo-Wan example.  The algorithm we present here is a modification of the Landsberg-Manivel algorithm \cite{Landsberg-Manivel04}, and uses ideas from \cite{FultonHarris, GoodWall, Ottaviani_5Lectures} and \cite{LandsbergTensorBook}. 
See \cite{OedingBates} for an example using this algorithm for 3-tensors.

First, we start with the partitions $(2,2,2)$ and $(3,3,3,3,3,3)$ associated (respectively) to the trivial representations of $\GL(3)$ and $\GL(6)$ in degrees 6 and 18, respectively. Then, we must find fillings of the associated tableaux so that the associated Young symmetrizer produces a non-zero image. 

After an exhaustive search, we found that the following pair of fillings will produce a non-zero image.
\[\ytableausetup{centertableaux,smalltableaux}
\ytableaushort{ac,be,df}%[[at,bt,dt],[ct,et,ft]]
\otimes
\ytableaushort{abc,abd,ade,bdf,cef,cef} %[[a, a, a, b, c, c], [b, b, d, d, e, e], [c, d, e, f, f, f]]
\;,\]
where, in the second filling, we use each letter three times indicating that we are parametrizing an invariant of degree 6 on $\bw{3}(W) \subset W^{\otimes 3}$.  We will use this filling to show how to construct the associated Young symmetrizer and compute its image.

The filling provides a recipe to construct a generic polynomial in terms of  auxiliary variables associated to the letters in the fillings by constructing matrices associated to the columns. For the filling
\[
\ytableaushort{ac,be,df} 
\;,\]
we associate the product of determinants
\[p_{V}=
\left| \begin{matrix}
a_1 & a_2 & a_3\\
b_1 & b_2 & b_3 \\
d_1 & d_2 & d_3
  \end{matrix} \right|
  \left| \begin{matrix}
c_1 & c_2 & c_3\\
e_1 & e_2 & e_3 \\
f_1& f_2 & f_3
  \end{matrix} \right|
.\]
Similarly, for the filling 
\[
\ytableaushort{abc,abd,ade,bdf,cef,cef}
\;,\]
we associate the product of determinants $p_{W}=$
 \[
\left| \begin{smallmatrix}
a_{11} & a_{12} & a_{13} &a_{14} & a_{15} & a_{16}\\
a_{21} & a_{22} & a_{23} &a_{24} & a_{25} & a_{26}\\
a_{31} & a_{32} & a_{33} &a_{34} & a_{35} & a_{36}\\
b_{31} & b_{32} & b_{33} &b_{34} & b_{35} & b_{36}\\
c_{21} & c_{22} & c_{23} &c_{24} & c_{25} & c_{26}\\
c_{31} & c_{32} & c_{33} &c_{34} & c_{35} & c_{36}
  \end{smallmatrix} \right|
  \left| \begin{smallmatrix}
b_{11} & b_{12} & b_{13} &b_{14} & b_{15} & b_{16}\\
b_{21} & b_{22} & b_{23} &b_{24} & b_{25} & b_{26}\\
d_{21} & d_{22} & d_{23} &d_{24} & d_{25} & d_{26}\\
d_{31} & d_{32} & d_{33} &d_{34} & d_{35} & d_{36}\\
e_{21} & e_{22} & e_{23} &e_{24} & e_{25} & e_{26}\\
e_{31} & e_{32} & e_{33} &e_{34} & e_{35} & e_{36}
  \end{smallmatrix} \right|
\left| \begin{smallmatrix}
c_{11} & c_{12} & c_{13} &c_{14} & c_{15} & c_{16}\\
d_{11} & d_{12} & d_{13} &d_{14} & d_{15} & d_{16}\\
e_{11} & e_{12} & e_{13} &e_{14} & e_{15} & e_{16}\\
f_{11} & f_{12} & f_{13} &f_{14} & f_{15} & f_{16}\\
f_{21} & f_{22} & f_{23} &f_{24} & f_{25} & f_{26}\\
f_{31} & f_{32} & f_{33} &f_{34} & f_{35} & f_{36}
  \end{smallmatrix} \right|.
\]
The next step is to extract the terms of the polynomial $p_{V}p_{W}$ one at a time and replace parts of the monomials with our target variables $x_{i,j,k,l}$, where $1\leq i \leq 3$ and $1\leq j< k< l \leq 6$. 

Let  the symbol $\lrcorner$ denote the contraction performed by ``taking the coefficient.'' For example, if we have a polynomial
\[
p = a_{1} b_{2}d_{3}c_{1}e_{3}f_{3} a_{11} a_{22} a_{33} b_{34} c_{25} c_{36} \cdot q
,\]
where $q$ does not depend on the variables $a$, then we can contract:
\[
(a_{1} a_{11} a_{22}a_{33})\lrcorner p =  b_{2}d_{3}c_{1}e_{3}f_{3} b_{34} c_{25} c_{36} \cdot q
.\]
We perform contractions to produce a polynomial in  $x_{i,j,k,l}$ that is in the image of the Young Symmetrizer associated to our initial fillings the algorithm in 
Figure~\ref{alg:Young}.

\begin{figure*}[!htb]\caption{An algorithm for evaluating Young symmetrizers}\label{alg:Young}
\begin{mdframed}
\begin{enumerate}
\item[]\hspace{-3em}\texttt{input:} $F=p_{V}p_{W}$ constructed as prescribed by the given fillings of tableaux.
\item[(a)] Replace $F$ with $\sum_{1\leq i\leq 3 \quad 1\leq j<k<l\leq 6} x_{i,j,k,l} \cdot \left(  a_{i} \cdot (a_{1j} \wedge a_{2k} \wedge a_{3l}) \right)\lrcorner F $, where the wedge notation indicates that we take the alternating sum over the permuted indices:
\[(a_{1j} \wedge a_{2k} \wedge a_{3l}) := 
\sum_{\sigma \in \mathfrak{S}_{3}}sgn(\sigma)a_{1\sigma(j)} a_{2\sigma(k)} a_{3\sigma(l)}
.\]
\item[(b)] Replace $F$ with 
$\sum_{\substack{1\leq i\leq 3 \\ 1\leq j<k<l\leq 6}} x_{i,j,k,l} \cdot \left(  b_{i} \cdot (b_{1j} \wedge b_{2k} \wedge b_{3l})\right)\lrcorner F.$
\item[(c-f)] Repeat step (b) for each letter $c,d,e,f$ playing the role of $b$.
\item[]\hspace{-3em}\texttt{output:} $F$, now a polynomial in $x_{i,j,k,l}$ in the image of the Young symmetrizer associated to the input filling of the Young tableaux.
\end{enumerate}
\end{mdframed}
\end{figure*}

To test whether this algorithm will produce a non-zero result, it is crucial to recognize that the procedure has a built-in evaluation option.  That is, at each step (a-f) in the algorithm in Figure~\ref{alg:Young}, one may evaluate the partial result at a fixed pre-determined point. The intermediate steps will consume much less memory and the evaluation will happen much more quickly than producing the polynomial and then evaluating it.  We used this method to find a filling that would produce a non-zero result and then, knowing that the filling we found would produce a non-zero polynomial, we applied the full algorithm to that filling.  We then check that the polynomial we produced is both non-zero (because it evaluates non-zero at at least one point of the ambient space) and vanishes on $\sigma_{5}(\PP^{2}\times \mathbb{G}(2,5))$ (because it vanishes on all parametrized points, i.e., on a Zariski open set).

\begin{theorem}\label{thm:sigma5}
The prime ideal of the hypersurface $\sigma_{5}(\PP^{2}\times \mathbb{G}(2,5))$ is generated by the single degree 6 polynomial (up to scale) constructed via 
the image of the Young symmetrizer associated to the filling
\[\ytableaushort{ac,be,df}%[[at,bt,dt],[ct,et,ft]]
\otimes
\ytableaushort{abc,abd,ade,bdf,cef,cef} %[[a, a, a, b, c, c], [b, b, d, d, e, e], [c, d, e, f, f, f]]
\;.\]
\end{theorem}
\begin{proof}
Let $F$ denote the polynomial resulting from the recipe given in the statement above. In particular, $F$ has precisely 10080 monomials, 5040 of which have coefficient $+1$ and 5040 of which have coefficient $-1$.  It can be downloaded from the ancillary files associated to the {\tt arXiv} version of this paper. One can check that $F$ vanishes on the irreducible Abo-Wan hypersurface $\sigma_{5}(\PP^{2}\times \mathbb{G}(2,5))$. The proof is complete if we can show that $F$ is irreducible.

We know that $F$ is non-zero, has degree 6, and is invariant under the $\SL(3)\times \SL(6)$ action.  
It is easy to check, in LiE for instance, that there are no non-trivial invariants of degree less than 6, and there is only one (up to scale) invariant in degree 6. If $F$ were to factor into factors of positive degree, the individual factors would define invariant hypersurfaces of lower degree.  Since this can't happen, $F$ is irreducible. Note this solves \cite[Problem 6.5]{AboWan}. 
\end{proof}

\begin{remark}We suppose that this equation may have an expression as a root of a determinant of a special matrix, similar to Ottaviani's degree 15 equation in \cite{Ottaviani09_Waring}, however our initial attempts at finding such an expression were unsuccessful. 
A natural guess is to 
start with $T \in V \otimes \wedge^3 W$
and use it to produce the $18 \times 36$ matrix 
$A_{T}\colon W \otimes W \rightarrow (V \otimes W)^* $, which has rank 3 when $T$ has rank 1 and rank $\leq 3r$ when $T$ has rank $r$.   
However, this map actually factors through a map
$ \wedge^2 W \rightarrow (V \otimes W)^* $
but this matrix is $18\times 15$ with maximum rank of~15.  This means that this construction cannot distinguish rank $5$ tensors from rank $6$~tensors.
\end{remark}

%%%%%%%
\section{The Abo-Wan hypersurfaces $\sigma_{3\ell+2}(\Seg(\PP^{2} \times \mathbb{G}(1,4\ell+2)))$ }\label{Sec:sigma3l+2}
In Abo and Wan's  study they identified an entire family of hypersurfaces:
\begin{theorem}[{\cite[Thm.~6.3]{AboWan}}]
The following secant varieties 
\begin{equation}\label{eq:family}
\sigma_{3\ell+2}(\Seg(\PP^{2} \times \mathbb{G}(1,4\ell+2))) \subset \PP \left( V\o \bw{2} W\right)  = \PP^{3\binom{4\ell + 3}{2}-1} 
\end{equation}
are hypersurfaces  for $\ell \geq 1$. %
 \end{theorem}
   For these secant varieties, Abo and Wan  \cite{AboWan}  used the exterior flattening construction (adapted from a construction by Ottaviani \cite{Ottaviani09_Waring}) to 
    produce a non-trivial equation that vanishes on them and shows that they are defective since these secant varieties are expected to fill their ambient spaces.  
   In particular, to a general tensor $T \in \CC^{3}\o \bw{2} \CC^{3\cdot(4\ell+3)}$, one associates  the $3\cdot(4\ell+3) \times 3\cdot(4\ell+3)$ exterior flattening matrix  $\varphi_{T}$, for which $\det \varphi_{T}$ is both nontrivial and vanishes~on~\eqref{eq:family}.    In addition, they bounded the dimension below by inductively selecting general points and showing that the tangent space has the claimed dimension.
They left it as an open problem to show that such polynomials are irreducible.  This is the missing ingredient to describing the generator of the corresponding prime ideal.
 
 \begin{remark}
Exterior flattening and variants (called Young flattenings) have also been used successfully to find equations for other secant varieties in a wide array of cases in \cite{LanOtt11_Equations}, and led to new results in complexity \cite{Landsberg_NewLower, LanOtt_NewLower}.  An analogous construction was used for partially symmetric tensors in \cite{CEO}, and for arbitrary tensors for the so-called ``salmon problem'' in \cite{Friedland2010_salmon,OedingBates, Friedland-Gross2011_salmon}.
\end{remark}

We consider the construction of this equation in the case when $\ell =1$. Here, $V = \CC^{3}$, (so $\bw{2} V \cong V^{*}$) and $W = \CC^{7}$.  For a tensor $T\in V\o \bw{2}W$ we can view $T$ as an element in $\bw{2}V^{*} \o \bw{2}W$, and associate to $T$ the natural linear map it induces:
\[
\varphi_{T}\colon V \o W^{*} \to V^{*}\o W
,\]
which is skew-symmetric in $W$ and (separately) skew-symmetric in $V$.  
The following provides an explicit construction 
of $\varphi_{T}$ in coordinates.

Choose a basis $a,b,c$ of $V$, and a basis $e_{i,j}$ of $\bw{2} W$. Then $\varphi_{T}$ is constructed from the $21\times 21$ Kronecker product of two matrices:
\[
\left(\begin{smallmatrix}
0 & a & -b \\
-a & 0 & c \\
b & -c & 0
\end{smallmatrix}\right)
\otimes \left(
\begin{smallmatrix}0 & e_{12} & e_{13} & e_{14} & e_{15} & e_{16} & e_{17}\\
      {-e_{12}} & 0 & e_{23} & e_{24} & e_{25} & e_{26} & e_{27}\\
      {-e_{13}} & {-e_{23}} & 0 & e_{34} & e_{35} & e_{36} & e_{37}\\
      {-e_{14}} & {-e_{24}} & {-e_{34}} & 0 & e_{45} & e_{46} & e_{47}\\
      {-e_{15}} & {-e_{25}} & {-e_{35}} & {-e_{45}} & 0 & e_{56} & e_{57}\\
      {-e_{16}} & {-e_{26}} & {-e_{36}} & {-e_{46}} & {-e_{56}} & 0 & e_{67}\\
      {-e_{17}} & {-e_{27}} & {-e_{37}} & {-e_{47}} & {-e_{57}} & {-e_{67}} & 0\\
      \end{smallmatrix}\right) \;.
\]
By replacing $a\o e_{jk}$ with $a_{jk}$ (similarly for $b\o e_{jk}$ and $c\o e_{jk}$), we obtain the (symmetric) matrix $\varphi_{T} =$
{\tiny \[
\scalemath{0.9}{
\left(
\makeatletter\makeatother\begin{smallmatrix}
0&0&0&0&0&0&0&0&a_{12}&a_{13}&a_{14}&a_{15}&a_{16}&a_{17}&0&-b_{12}&-b_{13}&-b_{14}&-b_{15}&-b_{16}&-b_{17}\\
0&0&0&0&0&0&0&-a_{12}&0&a_{23}&a_{24}&a_{25}&a_{26}&a_{27}&b_{12}&0&-b_{23}&-b_{24}&-b_{25}&-b_{26}&-b_{27}\\
0&0&0&0&0&0&0&-a_{13}&-a_{23}&0&a_{34}&a_{35}&a_{36}&a_{37}&b_{13}&b_{23}&0&-b_{34}&-b_{35}&-b_{36}&-b_{37}\\
0&0&0&0&0&0&0&-a_{14}&-a_{24}&-a_{34}&0&a_{45}&a_{46}&a_{47}&b_{14}&b_{24}&b_{34}&0&-b_{45}&-b_{46}&-b_{47}\\
0&0&0&0&0&0&0&-a_{15}&-a_{25}&-a_{35}&-a_{45}&0&a_{56}&a_{57}&b_{15}&b_{25}&b_{35}&b_{45}&0&-b_{56}&-b_{57}\\
0&0&0&0&0&0&0&-a_{16}&-a_{26}&-a_{36}&-a_{46}&-a_{56}&0&a_{67}&b_{16}&b_{26}&b_{36}&b_{46}&b_{56}&0&-b_{67}\\
0&0&0&0&0&0&0&-a_{17}&-a_{27}&-a_{37}&-a_{47}&-a_{57}&-a_{67}&0&b_{17}&b_{27}&b_{37}&b_{47}&b_{57}&b_{67}&0\\
0&-a_{12}&-a_{13}&-a_{14}&-a_{15}&-a_{16}&-a_{17}&0&0&0&0&0&0&0&0&c_{12}&c_{13}&c_{14}&c_{15}&c_{16}&c_{17}\\
a_{12}&0&-a_{23}&-a_{24}&-a_{25}&-a_{26}&-a_{27}&0&0&0&0&0&0&0&-c_{12}&0&c_{23}&c_{24}&c_{25}&c_{26}&c_{27}\\
a_{13}&a_{23}&0&-a_{34}&-a_{35}&-a_{36}&-a_{37}&0&0&0&0&0&0&0&-c_{13}&-c_{23}&0&c_{34}&c_{35}&c_{36}&c_{37}\\
a_{14}&a_{24}&a_{34}&0&-a_{45}&-a_{46}&-a_{47}&0&0&0&0&0&0&0&-c_{14}&-c_{24}&-c_{34}&0&c_{45}&c_{46}&c_{47}\\
a_{15}&a_{25}&a_{35}&a_{45}&0&-a_{56}&-a_{57}&0&0&0&0&0&0&0&-c_{15}&-c_{25}&-c_{35}&-c_{45}&0&c_{56}&c_{57}\\
a_{16}&a_{26}&a_{36}&a_{46}&a_{56}&0&-a_{67}&0&0&0&0&0&0&0&-c_{16}&-c_{26}&-c_{36}&-c_{46}&-c_{56}&0&c_{67}\\
a_{17}&a_{27}&a_{37}&a_{47}&a_{57}&a_{67}&0&0&0&0&0&0&0&0&-c_{17}&-c_{27}&-c_{37}&-c_{47}&-c_{57}&-c_{67}&0\\
0&b_{12}&b_{13}&b_{14}&b_{15}&b_{16}&b_{17}&0&-c_{12}&-c_{13}&-c_{14}&-c_{15}&-c_{16}&-c_{17}&0&0&0&0&0&0&0\\
-b_{12}&0&b_{23}&b_{24}&b_{25}&b_{26}&b_{27}&c_{12}&0&-c_{23}&-c_{24}&-c_{25}&-c_{26}&-c_{27}&0&0&0&0&0&0&0\\
-b_{13}&-b_{23}&0&b_{34}&b_{35}&b_{36}&b_{37}&c_{13}&c_{23}&0&-c_{34}&-c_{35}&-c_{36}&-c_{37}&0&0&0&0&0&0&0\\
-b_{14}&-b_{24}&-b_{34}&0&b_{45}&b_{46}&b_{47}&c_{14}&c_{24}&c_{34}&0&-c_{45}&-c_{46}&-c_{47}&0&0&0&0&0&0&0\\
-b_{15}&-b_{25}&-b_{35}&-b_{45}&0&b_{56}&b_{57}&c_{15}&c_{25}&c_{35}&c_{45}&0&-c_{56}&-c_{57}&0&0&0&0&0&0&0\\
-b_{16}&-b_{26}&-b_{36}&-b_{46}&-b_{56}&0&b_{67}&c_{16}&c_{26}&c_{36}&c_{46}&c_{56}&0&-c_{67}&0&0&0&0&0&0&0\\
-b_{17}&-b_{27}&-b_{37}&-b_{47}&-b_{57}&-b_{67}&0&c_{17}&c_{27}&c_{37}&c_{47}&c_{57}&c_{67}&0&0&0&0&0&0&0&0\\
\end{smallmatrix}\right)
.}\]
}

If $T$ has rank 1 as a tensor (up to the action of $\GL(3) \times \GL(7)$), we may assume that $T_{112}=1$ and all other coordinates are zero. In this case, $\varphi_{T}$ has  rank~4. The construction is linear in $T$, so if $T$ has rank $r$ then $\varphi_{T}$ has rank~$\leq 4r$ (because matrix rank is sub-additive).  In particular, if $T$ has rank~5, then $\varphi_{T}$ has rank $\leq 20$, so the determinant of $\varphi_{T}$ must vanish.
One checks that for random $T$, $\varphi_{T}$ has rank 21 so the $21\times 21$ determinant of $\varphi_{T}$ is non-trivial and produces the equation of  $\sigma_{5}(\Seg(\PP^{2} \times \mathbb{G}(1,6)))$. We verified these computations using {\tt Macaulay2} \cite{M2}.

The {\tt Bertini} computation described above that is summarized in Computation~\ref{thm:degree} indicates that (with high probability) this polynomial is irreducible.  A similar argument works for the cases $\ell =2,3$ as well. Given these numerical results, we were motivated to prove the following result (without the ``with high probability'' qualifier).
\begin{theorem}\label{thm:family} Let $V = \CC^{3}$ and $W = \CC^{4\ell + 3}$.
For each $\ell \geq 1$ the prime ideal of the irreducible hypersurface  
\[\sigma_{3\ell+2}(\Seg(\PP^{2} \times \mathbb{G}(1,4\ell+2))) \subset \PP \left( V\o \bw{2} W\right)  = \PP^{3\binom{4\ell + 3}{2}-1} \]
 is generated by the determinant of the $3(4\ell+3)\times 3(4\ell+3)$ matrix \[\varphi_{T}\colon V \o W^{*} \to V^{*}\o W.\]
\end{theorem}
\begin{proof}
We first explain how to construct the matrix $\varphi_{T}$ in general.
To that end, choose a basis $v_{1},v_{2},v_{3}$ of $V$, and a basis $e_{i,j}$ of $\bw{2} W$ and write  $E = (e_{i,j}) \in \bw{2}W$
which is a $(4\ell +3 )\times (4\ell +3)$ skew-symmetric matrix,
i.e., $E = (e_{i,j}) = -E^{t}$. 
Then, $\varphi_{T}$ is the $3(4\ell +3)\times 3(4\ell +3)$ matrix constructed via a  $\boxtimes$ product (see Section~\ref{sec:K}). Namely, we take the usual Kronecker product of matrices
\[\left(
\begin{smallmatrix}
0 & v_{1} & -v_{2} \\
-v_{1} & 0 & v_{3} \\
v_{2} & -v_{3} & 0
\end{smallmatrix}\right) \otimes E,
\]
 and replace each $v_{i}e_{j,k}$ with the variable $x_{ijk}$. The resulting matrix $\varphi_{T}$ represents a point $T \in V\o \bw{2} W \cong \bw{2}V^{*}\o\bw{2}W$.
Note, this variable replacement is crucial, because the identity  \eqref{detKr} implies that before our replacement of $v_{i}e_{j,k}$ with $x_{ijk}$, the determinant of the matrix we construct is zero. 
On the other hand, \cite[Lemma~4.1]{AboWan} provides tensor $T$ for which $\varphi_{T}$ has full rank.  In particular, $\det(\varphi_{T})\neq 0$.  Abo and Wan also explained why $\det \varphi_{T}$ vanishes on the appropriate secant variety, which is a consequence of the flattening construction. 

We will prove that the ideal generated by $\det \varphi_{T}$ is prime by showing that $\det \varphi_{T}$ is irreducible, which will be a consequence of Theorem~\ref{thm:boxproduct} below.
\end{proof}

\begin{lemma}\label{lem:degrees}
Suppose $V \cong \CC^{3}$ and $W \cong \CC^{s}$. An integer $d\leq 3 s$  can be the degree of an $\SL(V) \times \SL(W)$-invariant in $\CC[V \o \bw{2} W]$ only if
 $d  \equiv 0 \mod 3$  and
\[
d \in \{0,   \quad s/2, \quad s, \quad 3s/2,\quad 2s,  \quad 5s/2,   \quad3 s\} \cap \NN
.\]

\end{lemma}
\begin{proof}[Proof of lemma]
By Representation Theory \cite[Ch.~6]{LandsbergTensorBook}  or by considering the weights of \emph{isobaric} monomials  \cite[~Ch.4]{SturmfelsAlgorithms}, 
the invariants of degree $d$ in question are indexed by pairs of tableaux of sizes $3 \times \frac{d}{3}$ and $s \times \frac{2d}{s}$. In particular, $d$ must be divisible by $3$, and $2d$ must be divisible by $s$.
\end{proof}

Two-thirds of the cases of Theorem \ref{thm:family}, namely when $s = 4\ell+ 3$ and  $\ell \equiv 1,2 \mod 3$, follow directly from this lemma since here $s$ is odd, and not divisible by $3$, so the lowest degree of any invariant is $4\ell + 3$ in these cases.  
We know that $\det \varphi_{T}$ is non-zero by \cite[Lemma~4.1]{AboWan}. Since $\sigma_{3\ell+2}(\Seg(\PP^{2} \times \mathbb{G}(1,4\ell+2))) $ is an invariant irreducible hypersurface contained in an invariant hypersurface of minimal possible degree (defined by $\det \varphi_{T}$) then $\det \varphi_{T}$ must be irreducible.  The next two examples show that this argument is not sufficient for all cases.

\begin{example}
The case $\ell =0$ is the well-known $3\times 3$ determinantal hypersurface
\[\sigma_{2}(\Seg(\PP^{2}\times \mathbb{G}(1,2) ))\cong \sigma_{2} (\Seg(\PP^{2}\times (\PP^{2})^{*} ))
.\]
However, in this case the exterior flattening $\varphi_{T}\colon V \o W^{*} \to V^{*}\o W$ is $9\times 9$ and its determinant is the cube of the determinant of a generic $3\times 3$ matrix. 
\end{example}

%%%%%%%%%%%
\begin{example}
When $\ell =3$, $4\ell+3 = 15$,  $d = 15$, and $2d = 30$, 
we have a possible invariant given by a pair of a $3 \times 5$ tableau and a $15\times  2$ tableau. 
One checks, for example by a long computation in LiE, that the space of degree 15 invariants in on $\CC^3 \otimes \bw{2}\CC^{15}$ is one dimensional, so there is such an invariant.
And it could be that a degree 45 invariant factors as a product of an invariant of degree 15 and one of degree 30.  A more careful argument is needed to rule out this possibility.
\end{example}

%%%%%%%
\section{Determinants of tensor products of generic matrices}\label{sec:K}
In this section we make use of a simplified version of \textit{1-generic matrices} (see \cite{eisenbud1988linear}). If $P$ is a matrix filled with independent indeterminate entries, we will call $P$ \textit{1-generic} or  \textit{generic}. 
In this case we can view $P$ as an element of $A^{*}\o B\o X$, where $A$ and $B$ have dimensions $a$ and $b$ respectively, and $X$ is an $a\cdot b$ dimensional vector space with basis $\{x_{i,j}\}$, and as such we can think of $P$ as a linear mapping $A\to B$ that depends linearly on $X$. 

Let  $P = (p_{i,j})\in A^{*}\o B \o X$ and $Q = (q_{k,l}) \in C^{*}\o D \o Y$ be generic matrices and consider their tensor product
\[
P\boxtimes Q \in (A^{*}\o B \o X) \o (C^{*}\o D \o Y) =  A^{*}\o B \o C^{*}\o D \o (X \o Y)
,\]
which we view as a $4$-mode tensor with entries in $X\o Y$.  
We may flatten this tensor by collecting terms in the tensor product to obtain a generic matrix in $(A^{*}\o B) \o (C^{*}\o D) \o (X \o Y)$, thought of as a linear mapping 
\[
A\o B^{*} \to C^{*}\o D
,\]
depending linearly on $X\o Y$. In this case we ``vectorize'' both $P$ and $Q$ and take their tensor product, producing a rank-one matrix with entries linear in $X\o Y$. 

Another flattening is to view $P\boxtimes Q$ in $(A\o C)^{*}\o (B\o D) \o (X \o Y)$. 
In this flattening we see $P\boxtimes Q$ as a matrix with rows indexed by the double index $(i,k)$ and columns indexed by the double index $(j,l)$, and the entry in position $((i,k), (j,l))$ is the tensor product of variables $p_{i,j}\o q_{k,l} \in X\o Y$. 
Note the usual Kronecker product of matrices (also denoted by the tensor product symbol $\o$) would put the symmetric product $p_{i,j}q_{k,l}$ in that position.  More specifically, the Kronecker product $P\o Q$  is an element of $(A\o C)^{*}\o (B\o D) \o (X \circ Y)$, where $X \circ Y$ may be viewed as a space of bilinear quadratic polynomials.  On the other hand we view $X\o Y$ as the space of (non-symmetric) bilinear forms. The difference between these two spaces only becomes apparent when considering polynomials on them of degree $\geq 2$.  In particular, the determinant of the Kronecker product satisfies the well-known property
for square matrices $P$ and $Q$ of sizes $m$ and $n$ respectively,
\begin{equation}\label{detKr}
\det(P\o Q) = \det(P)^{n} \det(Q)^{m} \quad \in \mathrm{Sym}_{n}(X\circ Y)
. \end{equation}
This property implies that if either $P$ or $Q$ is rank-deficient, then so is $P\o Q$. We will be primarily interested in the case when $P$ is a generic $3\times 3$ skew-symmetric matrix, and as such the determinant of its Kronecker product with any other matrix is zero because of \eqref{detKr}.
To the contrary,  the $\boxtimes$ product of two generic matrices usually produces a full-rank generic matrix. So we use the symbol~$\boxtimes$ to make the distinction between it and the Kronecker product.
We are led to study the irreducibility of the determinant
\[
\det(P\boxtimes Q) \quad \in \mathrm{Sym}_{n}(X\o Y)
.\]
For what follows, we abbreviate the notation for generic matrices, not explicitly naming the spaces of variables on which the matrices depend.

\begin{theorem}\label{thm:boxproduct}
Let $P$ and $Q$ be  $3\times 3$ and $s\times s$ generic skew-symmetric matrices.
\begin{enumerate}
\item If $s=1$ or $s = 2$, then $\det(P \boxtimes Q) = 0$.
\item If $s=3$, then $\det(P \boxtimes Q)$ factors as the cube of a cubic polynomial.
\item If $s=4$, then $\det(P \boxtimes Q)$ factors as the square of a sextic polynomial.
\item If $s\geq 5$, then $\det(P \boxtimes Q)$ is irreducible.
\end{enumerate}
\end{theorem}
\begin{proof}
Note that a large set of cases are covered by Lemma~\ref{lem:degrees}, namely when there is no integer $d$ satisfying the condition that $d$ is divisible by $3$, $2d$ is divisible by $s$ and $d<3s$. In those cases $\det(P \boxtimes Q)$ is the lowest degree invariant, so it cannot factor into a product of lower degree invariants. For the rest of the cases we need a more in-depth argument.
 
 \textbf{Summary of proof:} We first handle all small cases ($s\leq 18$) by direct computation. 
Then we proceed by induction.   We will show that if $\det(P\boxtimes Q)$ has a factorization as a product of non-trivial invariants this will force a non-trivial factorization of a $P\boxtimes Q'^{c}$, where $Q'^{c}$ is a skew-symmetric matrix of size $(s-3)\times (s-3)$, which can't happen by induction. 
 
The case $s=1$ is trivial because the determinant is just the determinant of $P$ in renamed variables. The cases $s=2,3,4$ are easy to verify in {\tt Macaulay2} directly by constructing the usual tensor product matrix, substituting new variables $x_{i,j,k,l}$ for  $p_{i,j}q_{k,l}$, and using the \defi{factor} command.  As $s$ grows, this computation becomes much more difficult.

For $s = 5,6,\dots,18$ we specialized the variables in the matrix $P \boxtimes Q$ to a random line, computed the determinant, and checked that the resulting homogeneous polynomial in 2 variables had the same degree and did not factor over $\QQ$. This provides 
a certificate that the original polynomial is irreducible.
Note, if the specialized polynomial were to factor, the test would be inconclusive.

For $s\geq 19$ we proceed by induction.
Let $P$ be a generic $3\times 3$ skew-symmetric matrix, $P \in \bw{2}A$ with $A\cong \CC^{3}$, and let $Q$ be a generic $s\times s$ skew-symmetric matrix, $Q\in \bw{2}B$, with $B\cong \CC^{s}$. 
Suppose $B = B' \oplus B'^{c}$ is a splitting with $\dim B'=3, \; \dim B'^{c} =s-3$. 
 Let $Q'$ denote the first $3\times 3$ principal submatrix of $Q$ and let $Q'^{c}$ denote the principal minor of $Q$ with~complementary~indices, which is necessarily the last $(s-3)\times(s-3)$ principal minor of $Q$.  
 In particular, $Q'\in \bw{2}B'$ and $Q'^{c}\in \bw{2}B'^{c}$. Now (by construction)
$P\boxtimes Q' \; \in(A \o B') \o (A\o B')$ and $P\boxtimes Q'^{c} \; \in (A \o B'^{c}) \o (A\o B'^{c})$ 
are complementary principal minors of $P\boxtimes Q$, respectively of size $9\times 9$ and $3(s-3)\times 3(s-3)$.

Consider the following sequence of ring homomorphisms:
\[ \CC[P\boxtimes Q] \to\CC[(P\boxtimes Q')\oplus (P\boxtimes Q'^{c})], 
\quad \quad \text{coordinate projection,}\] 
\[\CC[(P\boxtimes Q')\oplus (P\boxtimes Q'^{c})] \to  \CC[P\boxtimes Q'^{c}], \quad \quad \text{evaluation at a point $ T' \in \bw{2}A \o \bw{2}B'$,}
\]
 and let $\xi_{T'}$ denote their composition.
We may choose $T'$ randomly so that $\xi_{T'}(\det(P\boxtimes Q') =: C$ with $C\neq 0$. 
Also note that
\[\det (\xi_{T'}(P\boxtimes Q)) = C\cdot  \det(P\boxtimes Q'^{c}) 
\]is a non-zero irreducible polynomial of degree $3(s-3)$ by the induction hypothesis.

For contradiction, suppose $\det(P\boxtimes Q) = f\cdot g$ with both $f$ and $g$ non-constant invariants.
After the evaluation $\xi_{T'}$ we have
\begin{equation}\label{smaller}
\xi_{T'}(f\cdot  g) = \xi_{T'}(f) \cdot \xi_{T'}(g) = 
 C\cdot \det(P\boxtimes Q'^{c})
,\end{equation}
The right-most side of \eqref{smaller} is non-zero and irreducible as long as $s-3 \geq 5$ by the induction hypothesis, so either $\xi_{T'}(f)$ or $\xi_{T'}(g)$ must be a (non-zero) constant.

The degrees of $ f$  and $ g$ must be positive integers satisfying the conditions in Lemma~\ref{lem:degrees}
(if this is impossible then we could end the proof earlier).  Since the evaluation $\xi_{T'}$ reduced the degree of $\det(P\boxtimes Q)$ by $9$ to obtain  \eqref{smaller}, the only way for one of $\xi_{T'}(f)$ or $\xi_{T'}(g)$ to be constant would be if  there were a positive integer $s$ satisfying one of the following equations: 
\[
\frac{s}{2}-e=0,\quad 
s-e = 0 , \quad
\frac{3s}{2}-e =0, \quad
2s-e=0, \quad
\frac{5s}{2}-e =0, \quad
3s -e =0
, \]
for some integer $e$ with $0\leq e \leq 9$. 
The maximum $s$ for which there is a possible solution to any of these equations is when $s=18$ and $e=9$. 
Since we assumed $s\geq 19$ this would imply that both $\xi_{T'}(f)$ and $\xi_{T'}(g)$ are non-constant.
This contradiction concludes the proof.
\end{proof}

\begin{remark}
One may re-interpret Theorem~\ref{thm:boxproduct} in light of projective duality as follows.
When the dual of the Segre-Grassmann variety is a hypersurface, its equation is a type of hyperdeterminant. 
One may ask if that hyperdeterminant could specialize to the equation of one of the hypersurfaces in our study.  Tocino-Sanchez's recent solution \cite{Tocino} to Ottaviani's open question \#2 in \cite{Ottaviani_Hyperdeterminants} (which is the skew-symmetric version of a problem on hyperdeterminants considered in \cite{Oeding_hyperdet}), indicates that our equations cannot be the specialization of the usual hyperdeterminant (the equation of the dual of a Segre product).  We still wonder about a possible connection between determinants of exterior flattenings and duals of Segre-Grassmann varieties.
\end{remark}

\section*{Acknowledgements}
We thank Hirotachi Abo and Giorgio Ottaviani 
for discussing these problems with us,
their encouragement, and useful suggestions.  
\bibliographystyle{amsplain}

\frenchspacing

\bibliography{main_bibfile}

\end{document}